\documentclass[12pt]{article}
\usepackage{a4}
\usepackage{amsthm}
\usepackage{amsfonts}
\usepackage{amssymb}
\usepackage{amsmath}
\usepackage{cite}
\usepackage{epsfig}
\usepackage{xcolor}

\newcommand\NN{{\mathbb N}}
\newcommand\RR{{\mathbb R}}

\newcommand\cut[1]{\|#1\|_{\square}}
\newcommand{\Sp}{\operatorname{Sp}}
\newcommand{\Tr}{\operatorname{Tr}}
\newcommand{\Ker}{\operatorname{Ker}}
\newcommand{\Aut}{\operatorname{Aut}}
\newcommand{\dd}{\;\mathrm{d}}
\newcommand{\Id}{\operatorname{Id}}

\newtheorem{theorem}{Theorem}

\newtheorem{lemma}[theorem]{Lemma}

\DeclareMathOperator{\op}{op}

\DeclareTextCompositeCommand{\v}{OT1}{l}{l\nobreak\hspace{-.1em}'}
\DeclareTextCompositeCommand{\v}{OT1}{t}{t\nobreak\hspace{-.1em}'\nobreak\hspace{-.15em}}

\begin{document}
\title{Convergence of spectra of digraph limits}
\author{Jan Greb\'\i{}k\thanks{Institute of Mathematics, Leipzig University, Augustusplatz 10, 04109 Leipzig, Germany. E-mail: {\tt grebikj@gmail.com}. Previous affiliation: Faculty of Informatics, Masaryk University, Botanick\'a 68A, 602 00 Brno, Czech Republic. This author was supported by by MSCA Postdoctoral Fellowships 2022 HORIZON-MSCA-2022-PF-01-01 project BORCA grant agreement number 101105722.}\and
        Daniel Kr{\'a}\v l\thanks{Institute of Mathematics, Leipzig University, Augustusplatz 10, 04109 Leipzig, and Max Planck Institute for Mathematics in the Sciences, Inselstra{\ss}e 22, 04103 Leipzig, Germany. E-mail: {\tt daniel.kral@uni-leipzig.de}. Previous affiliation: Faculty of Informatics, Masaryk University, Botanick\'a 68A, 602 00 Brno, Czech Republic. This author was supported by the Alexander von Humboldt Foundation in the framework of the Alexander von Humboldt Professorship endowed by the Federal Ministry of Education and Research.}\and
	Xizhi Liu\thanks{School of Mathematical Sciences, University of Science and Technology of China, Hefei, China. E-mail: {\tt liuxizhi@ustc.edu.cn}. Previous affiliation: Mathematics Institute and DIMAP, University of Warwick, Coventry CV4 7AL, United Kingdom. This author was supported by the ERC Advanced Grant 101020255 FDC.}\and
        Oleg Pikhurko\thanks{Mathematics Institute and DIMAP, University of Warwick, Coventry CV4 7AL, United Kingdom. E-mail: {\tt O.Pikhurko@warwick.ac.uk}. This author was supported by the ERC Advanced Grant 101020255 FDC.}\and
	Julia Slipantschuk\thanks{Department of Mathematics, University of Bayreuth, D-95440 Bayreuth, Germany. This author was supported by ERC-Advanced Grant 833802 Resonances.}}

\date{} 
\maketitle

\begin{abstract}
The relation between densities of cycles and the spectrum of a graphon,
which implies that the spectra of convergent graphons converge,
fundamentally relies on the self-adjointness of the linear operator associated with a graphon.
In this short paper,
we consider the setting of digraphons, which are limits of directed graphs, and
prove that the spectra of convergent digraphons converge.
Using this result, we establish the relation between densities of directed cycles and the spectrum of a digraphon.
\end{abstract}

\section{Introduction}
\label{sec:intro}

Spectral methods play an important role in combinatorics and particularly in graph theory~\cite{GodR01}.
It is well known that
the trace of the $\ell$-th power of the adjacency matrix of an (undirected or directed) graph $G$
is equal to the number of closed walks of length $\ell$ in $G$.
Consequently, the number of closed walks of length $\ell$
is equal to the sum of the $\ell$-th powers of the eigenvalues of $G$,
i.e., the eigenvalues of its adjacency matrix.

The theory of combinatorial limits provides analytic tools to represent and to study large combinatorial objects.
In the setting of undirected graphs,
a large graph is represented by an analytic object called a \emph{graphon},
which is a symmetric measurable function $W:[0,1]^2\to [0,1]$.
In this setting, we speak of a density of a graph $G$ in $W$,
which means the (limit) density of $G$ in large graphs represented by $W$;
we refer to Section~\ref{sec:prelim} for further details and
to the monograph~\cite{Lov12} for a comprehensive treatment of the matter.
The spectrum of $W$ viewed as a Hilbert-Schmidt integral operator is either finite or countably infinite and
the sum of the $\ell$-th powers of the non-zero eigenvalues of $W$ (counting their geometric multiplicities)
is equal~\cite[Section 6]{BorCLSV12} to the density of cycles of length $\ell$ in the graphon $W$,
see also~\cite[Section 7.5]{Lov12}.
The proofs from~\cite{BorCLSV12,Lov12} of this equality fundamentally exploit the fact that $W$ is a self-adjoint operator.
This equality between the density of cycles and the sum of the $\ell$-th powers of non-zero eigenvalues of a graphon
implies that
the normalized spectra of any sequence of graphs $G$ converging to a graphon $W$
converge to the spectrum of $W$~\cite[Section 6]{BorCLSV12}, see also~\cite[Section 11.6]{Lov12}.

In this short paper, we consider the asymmetric setting of directed graphs
where the methods used in the symmetric setting of undirected graphs fail to work
because the associated operators need not be self-adjoint.
The motivation for considering this setting comes from applications in extremal combinatorics.
In the setting of tournament limits, which have been initially studied in~\cite{ChaGKN19,Tho18,ZhaZ20},
spectral properties of large tournaments were studied in relation to maximizing the density of directed cycles
in~\cite{GrzKLV23,MaT22}.
It can be shown~\cite{GrzKLV23} that
the normalized spectra of any convergent sequence of tournaments converge, however, 
the link between the limit of the normalized spectra and the spectrum of the limit object (tournamenton) has been missing.
We employ methods from functional analysis to establish this link (Theorem~\ref{thm:spectrum} in Section~\ref{sec:spectrum}), and
we then use this link to show that
the sum of the $\ell$-th powers of the non-zero eigenvalues of a digraphon $W$ (counting their algebraic multiplicities)
is equal to the density of directed cycles of length $\ell$ in the digraphon $W$ (Theorem~\ref{thm:cycle} in Section~\ref{sec:cycle}).
If proven earlier,
the latter would make the spectral arguments used in~\cite{GrzKLV23,MaT22} more straightforward,
see also~\cite[Section~1]{HlaS25} for the discussion of this matter, as
it would be possible to work directly with the spectrum of the limit tournamenton rather than the limit of the multisets
representing the spectra of tournaments.

\section{Preliminaries}
\label{sec:prelim}

In this section,
we briefly overview the notation used throughout the paper and
review basic properties of digraph limits.
We start with some general notation.
The set of the first $k$ positive integers is denoted by $[k]$.
If $X$ is a set of complex numbers and $z$ is a complex number,
then $zX$ is the set $\{zx:x\in X\}$ and, if $z$ is non-zero, $X/z$ is the set $\{x/z:x\in X\}$.
If $x\in\mathbb{C}$ and $\epsilon$ is a positive real,
we write $\mathcal{B}_\epsilon(\lambda)$ for the open $\epsilon$-neighborhood of $\lambda$.
We say that a sequence $(X_n)_{n\in\NN}$ of subsets of a metric space converges to a set $X$ in the \emph{Hausdorff sense}
if for every $\varepsilon>0$, there exists $n_0$ such that the following holds for every $n\ge n_0$:
\begin{itemize}
\item for every $x\in X_n$, there exists $x'\in X$ at distance at most $\varepsilon$ from $x$, and
\item for every $x\in X$, there exists $x'\in X_n$ at distance at most $\varepsilon$ from $x$.
\end{itemize}
We remark that the limit set $X$ is not unique unless $X$ is required to be closed.

Directed graphs, \emph{digraphs} for short, that are considered throughout this paper
contain neither parallel edges (directed either way) nor loops unless stated otherwise.
A particular digraph is the cyclically oriented cycle of length $\ell$, which is denoted by $C_{\ell}$.
If $G$ is a digraph, then $|G|$ denotes the number of vertices of $G$, and
$V(G)$ and $E(G)$ the vertex set and the edge set of $G$, respectively.
The \emph{adjacency matrix} $A_G$ of $G$ is a $|G|\times |G|$ square matrix with rows and columns labeled by the vertices of $G$;
the entry of $A_G$ in the row labeled by a vertex $u$ and in the column labeled by a vertex $v$
is equal to $1$ if and only if $G$ contains an edge directed from $u$ to $v$.
The \emph{spectrum} of a digraph $G$ is the spectrum of its adjacency matrix $A_G$ and
is denoted by $\Sp(G)$;
we write $m_G(\lambda)$ for the algebraic multiplicity of the eigenvalue $\lambda\in\Sp(G)$.
The \emph{normalized spectrum} of $G$ is $\Sp(G)/|G|$.

The \emph{subgraph density} of a digraph $H$ in a digraph $G$, which is denoted by $d(H,G)$,
is the probability that $|H|$ randomly chosen vertices of $G$ induce a subgraph isomorphic to $H$;
if $|H|>|G|$, we set $d(H,G)$ to $0$.
A sequence $(G_n)_{n\in\NN}$ of digraphs is \emph{convergent}
if $|G_n|$ tends to infinity and the sequence $d(H,G_n)$ converges for every $H$.
In Section~\ref{sec:cycle}, we will need the notion of the homomorphism density defined as follows:
the \emph{homomorphism density} of a digraph $H$ in a digraph $G$, which is denoted by $t(H,G)$,
is the probability that a uniformly chosen random mapping $f$ from $V(H)$ to $V(G)$ is a homomorphism,
i.e., $f(u)f(v)\in E(G)$ whenever $uv\in E(H)$.
It is easy to show, see e.g. \cite[Section~5.2.3]{Lov12}, that a sequence $(G_n)_{n\in\NN}$ of digraphs is convergent if and only if
$|G_n|$ tends to infinity and $t(H,G_n)$ converges for every digraph $H$.
Observe (recall that $C_{\ell}$ is the cyclically oriented cycle of length $\ell$) that
\[t(C_\ell,G)|G|^{\ell}=\Tr A_G^{\ell}=\sum_{\lambda\in\Sp(G)}m_G(\lambda)\lambda^{\ell}.\]

Convergent sequences of digraphs can be represented by an analytic object,
which is analogous to graphons in the setting of graphs studied in~\cite{BorCL10,BorCLSSV06,BorCLSV08,BorCLSV12,LovS06,LovS10}.
We remark that limits of tournaments were used earlier in~\cite{ChaGKN19,Tho18,ZhaZ20},
see also~\cite{DiaJ08,LovS10i} for related concepts.
A \emph{kernel} is a bounded measurable function $W:[0,1]^2\to\RR$ and
a \emph{digraphon} is a non-negative kernel $W$ such that $W(x,y)+W(y,x)\le 1$ for all $(x,y)\in [0,1]^2$.

If $W$ is a digraphon,
then the \emph{$n$-vertex $W$-random digraph} is obtained as follows:
choose $x_1,\ldots,x_n$ uniformly in $[0,1]$ and
include an edge directed from the $i$-th vertex to the $j$-th vertex with probability $W(x_i,x_j)$,
an edge directed from the $j$-th vertex to the $i$-th vertex with probability $W(x_j,x_i)$, and
no edge between the $i$-th vertex and the $j$-th vertex with probability $1-W(x_i,x_j)-W(x_j,x_i)$,
with all choices being mutually independent.
The \emph{subgraph density} of a digraph $H$ in a digraphon $W$, which is denoted by $d(H,W)$,
is the probability that a $|H|$-vertex $W$-random digraph is isomorphic to $H$.
In particular, it holds that
\[
\scalebox{0.85}{$
d(H,W)=\frac{|V(H)|!}{|\Aut(H)|}\int\limits_{[0,1]^{V(H)}}\prod_{uv\in E(H)}W(x_u,x_v)\prod\limits_{\substack{uv\not\in E(H)\\vu\not\in E(H)}}\left(1-W(x_u,x_v)-W(x_v,x_u)\right)\dd x_{V(H)}$},
\]
where $\Aut(H)$ is the automorphism group of $H$.
The definition of the \emph{homomorphism density} extends to digraphons as follows:
\begin{equation}
t(H,W)=\int\limits_{[0,1]^{V(H)}}\prod_{uv\in E(H)}W(x_u,x_v)\dd x_{V(H)}.
\label{eq:tHW}
\end{equation}
\noindent If $(G_n)_{n\in\NN}$ is a convergent sequence of digraphs,
we say that $W$ is a limit of $(G_n)_{n\in\NN}$ if $d(H,G_n)$ converges to $d(H,W)$ for every digraph $H$.
It is easy to show that
if a sequence $(G_n)_{n\in\NN}$ of digraphs is convergent,
then $W$ is a limit digraphon of the sequence $(G_n)_{n\in\NN}$
if and only if $t(H,W)$ is the limit of $t(H,G_n)$ for every digraph $H$.
The same argument as in the case of (undirected) graphs, e.g. as in~\cite[Theorem~2.2]{LovS06}, yields that
there exists a limit digraphon for every convergent sequence of digraphs.
In the other direction, a sequence of $W$-random digraphs with increasing number of vertices
is convergent and $W$ is its limit with probability one.

Two digraphons $W_1$ and $W_2$ are \emph{weakly isomorphic}
if $d(H,W_1)=d(H,W_2)$ for every digraph $H$,
i.e., $W_1$ and $W_2$ are limits of the same set of convergent sequences of digraphs.
If $W$ is a digraphon and
$\varphi:[0,1]\to [0,1]$ is a measure preserving map (unless stated otherwise, we always consider the Lebesgue measure),
then $W^\varphi$ is the digraphon defined as $W^\varphi(x,y)=W(\varphi(x),\varphi(y))$;
note that the digraphons $W$ and $W^\varphi$ are weakly isomorphic.
It can be shown as in the case of (undirected) graphs~\cite[Theorem~2.1]{BorCL10} that
if $W_1$ and $W_2$ are weakly isomorphic digraphons,
then there exist measure preserving maps $\varphi$ and $\psi$ such that
$W_1^\varphi$ and $W_2^\psi$ are equal almost everywhere.

A digraphon $W$ is a \emph{$k$-step} digraphon if there exists a partition of $[0,1]$ into $k$ measurable sets $J_1,\ldots,J_k$ such that
$W$ is constant on $J_i\times J_j$ for all $i,j\in [k]$.
If $G$ is a $k$-vertex digraph, we can associate with $G$ a $k$-step digraphon $W$ such that
each set $J_i$, $i\in [k]$, is an interval of length $1/k$ and
the value of $W$ on $J_i\times J_j$ is equal to the entry of the adjacency matrix $A_G$ in the $i$-th row and the $j$-th column.
We define the \emph{cut norm} of a kernel $W$ as
\[\cut{W}=\sup_{g,h: [0,1]\to [0,1]}\left|\int\limits_{\ [0,1]^2} g(x)W(x,y)h(y)\dd x\dd y\right|,\]
where the supremum is taken over measurable functions $g,h: [0,1]\to [0,1]$.
Equivalently, the cut norm can be defined using the supremum over characteristic functions only,
i.e., measurable functions $g,h:[0,1]\to\{0,1\}$;
it is easy to check that these two definitions are equivalent,
see e.g.~\cite[Lemma~8.10]{Lov12} for the analogous argument in the case of symmetric kernels.
We remark that $\cut{W}=0$ if and only if $W$ is equal to $0$ almost everywhere.
This defines the \emph{cut metric} on the space of digraphons by setting $d_{\square}(W_1,W_2)=\cut{W_1-W_2}$.
The \emph{cut distance} of two digraphons $W_1$ and $W_2$ is defined as
\[\delta_{\square}(W_1,W_2)=\inf_{\varphi} d_{\square}(W_1,W_2^\varphi)\]
where the infimum is taken over all measure preserving maps $\varphi:[0,1]\to [0,1]$.
Following the steps of the corresponding proof for graphons (given in e.g.~\cite[Theorem~11.3]{Lov12}),
it can be shown that if $(G_n)_{n\in\NN}$ is a convergent sequence of digraphs,
then the sequence of the associated step digraphons converges in the cut distance.
Vice versa, if $(G_n)_{n\in\NN}$ is a sequence of digraphs with $|G_n|$ tending to infinity and
the sequence of the associated step digraphons converges in the cut distance,
then the sequence $(G_n)_{n\in\NN}$ is convergent.

\section{Spectrum convergence}
\label{sec:spectrum}

We can associate a digraphon, or more generally a kernel, $W$
with a Hilbert-Schmidt integral operator $T_W$ on the complex Hilbert space $L_2[0,1]$,
where $[0,1]$ is endowed with the Lebesgue measure.
In particular, the operator $T_W$ associated with a kernel $W$ is defined for $f\in L_2[0,1]$ as
\[T_W(f)(x)=\int_{[0,1]} W(x,y) f(y) \dd y.\]
\noindent In what follows, we will use $W$ both for a kernel and the associated operator, which is denoted $T_W$ above;
we believe that this simplifies the notation while there is no risk of confusion.

We now recall some basic definitions and facts from functional analysis
while referring to the monograph by Ahues, Largillier and Limaye~\cite{AhuLL01}.
The operator norm for operators on $L_2[0,1]$ is denoted by $\|{-}\|_{2\to 2}$, and
$\Id$ denotes the identity operator.
For any kernel $W$,
the operator $W$ is compact \cite[Section~7.5]{Lov12}, hence bounded.
The spectrum $\Sp(W)$ is the set of all $\lambda\in \mathbb{C}$ such that
$W-\lambda\Id$ does not have a bounded inverse.
%the bounded operator $W-\lambda\Id$ is not invertible, or equivalently, using the open mapping theorem, does not have a bounded inverse.
Since $W$ is a compact operator,
every \emph{non-zero} point $\lambda$ of $\Sp(W)$ is isolated and an eigenvalue, i.e., $\dim \Ker \left(W-\lambda\Id\right)>0$,
see~\cite[Remark~1.34]{AhuLL01}.
The \emph{algebraic multiplicity} of a non-zero eigenvalue $\lambda$ of $W$, denoted by $m_W(\lambda)$, is
\[m_W(\lambda)=\lim_{k\to\infty} \dim \Ker \left( (W-\lambda\Id)^k\right).\]
We remark that
this definition of the algebraic multiplicity for a non-zero eigenvalue $\lambda$ in the setting of digraphons
is equivalent to the standard definition of the algebraic multiplicity of an isolated spectral point of an operator
as the rank of the respective Riesz projection (see~\cite[Chapter~1]{AhuLL01}).
The compactness of $W$ and \cite[Proposition~1.31]{AhuLL01} yield that
every non-zero point $\lambda$ of $\Sp(W)$ has finite algebraic multiplicity, see also~\cite[Remark 1.34]{AhuLL01}.
Since the operator $W$ does not admit a bounded inverse as $L_2[0,1]$ is infinite-dimensional and $W$ is compact,
the value $0$ is contained in $\Sp(W)$, but it need not be an eigenvalue.
It can be shown that if $W$ is a digraphon and $\varphi:[0,1]\to [0,1]$ a measure preserving map,
then $\Sp(W)=\Sp(W^\varphi)$ and
$m_W(\lambda)=m_{W^\varphi}(\lambda)$ for every $\lambda\in\Sp(W)\setminus\{0\}$.

In our argument,
we need a statement on the convergence of spectra of operators on complex Banach spaces with algebraic multiplicities,
which we formulate as a lemma.
In order to do that, we need to recall a definition of a less standard mode of convergence introduced in \cite[Chapter~2.1]{AhuLL01}.
A sequence $(T_n)_{n\in\NN}$ of bounded linear operators on a complex Banach space \emph{$\nu$-converges} to a linear operator $T$
if their operator norms are uniformly bounded,
\[\lim_{n\to\infty}\|(T_n-T)T\|_{\op}=0 \qquad\mbox{and}\qquad \lim_{n\to\infty}\|(T_n-T)T_n\|_{\op}=0.\]
The following statement can be found in~\cite[Corollary 2.13(ii)]{AhuLL01}.

\begin{lemma}
\label{lm:cor213}
Let $(T_n)_{n\in\NN}$ be a sequence of bounded linear operators on a complex Banach space that $\nu$-converges to $T$.
The following holds for every non-zero $\lambda\in\Sp(T)$ that
is an isolated point of $\Sp(T)$ and has finite algebraic multiplicity.

For any $\varepsilon>0$ such that the $\Sp(T)\cap B_{2\varepsilon}(\lambda)=\{\lambda\}$,
there exists $n_0$ such that it holds for every $n\ge n_0$ that
\[\Sp(T_n)\cap B_{\varepsilon}(\lambda)\not=\emptyset
  \qquad\mbox{and}\qquad
  m_T(\lambda)=\sum_{\lambda'\in \Sp(T_n)\cap B_{\varepsilon}(\lambda)}m_{T_n}(\lambda').\]
\end{lemma}

In Lemma~\ref{lm:converge}, we will show that every convergent (in the cut metric) sequence of digraphons is $\nu$-convergent.
We remark that a more well-known version of Lemma~\ref{lm:cor213},
see also~\cite[Corollary 2.13(i)]{AhuLL01},
holds under the assumption that $T_n$ converges to $T$ in $\|{-}\|_{\op}$.
Janson~\cite[Lemma~E.6]{Jan13} proved that
every convergent (in the cut metric) sequence of \emph{graphons} is $\|{-}\|_{2\to 2}$-convergent;
the argument also extends to the case of digraphons~\cite{JanPC}
but since the details are not publicly available,
we include a self-contained argument using the weaker notion of $\nu$-convergence.
We also remark that the extension of~\cite[Lemma~E.6]{Jan13} to digraphons was claimed in~\cite{Hon20}, however,
we were not able to verify the correctness of all claims in~\cite{Hon20},

To prove Lemma~\ref{lm:converge}, we need the following auxiliary statement.

\begin{lemma}
\label{lm:cutnu}
Let $V$ and $U$ be kernels such that $\|V\|_\infty\le 1$ and $\|U\|_\infty\le 1$.
Then it holds that
\[\|VU\|_{2\to 2}\le 2\cut{V}^{1/2}.\]
\end{lemma}

\begin{proof}
We will show that
\begin{equation}
\|VUf\|_2\le 2\cut{V}^{1/2}
\label{eq:cutnu}
\end{equation}
for any function $f:[0,1]\to\RR$ with $\|f\|_2=1$.
Note that 
since the values of $V$ and $U$ are real,
\eqref{eq:cutnu} implies that
\begin{align*}
\|VU(f+\iota g)\|_2^2
& = \|VUf\|_2^2+\|VUg\|_2^2\\
& \le 4\cut{V}\;\|f\|_2^2+4\cut{V}\;\|g\|_2^2\\
& = 4\cut{V}\;\|f+\iota g\|_2^2
\end{align*}
for any functions $f:[0,1]\to\RR$ and $g:[0,1]\to\RR$ (we use $\iota$ for the imaginary unit),
which will yield the statement of the lemma.

We now prove \eqref{eq:cutnu}.
Fix a function $f:[0,1]\to\RR$ such that $\|f\|_2=1$.
Since $\|U\|_\infty\le 1$, it holds that $\|Uf\|_{\infty}\le \|f\|_1\le \|f\|_2= 1$;
similarly, it holds that $\|VUf\|_{\infty}\le 1$.
We next define auxiliary functions $g_1$, $g_2$, $h_1$ and $h_2$ from $[0,1]$ to $[0,1]$ as follows:
\begin{align*}
g_1(x) & = \max\{0,(Uf)(x)\},\\
g_2(x) & = \max\{0,-(Uf)(x)\},\\
h_1(x) & = \max\{0,(VUf)(x)\}, \mbox{ and}\\
h_2(x) & = \max\{0,-(VUf)(x)\}.
\end{align*}
Observe that
\[\|VUf\|_2^2=\int\limits_{[0,1]^2} (h_1-h_2)(x)V(x,y)(g_1-g_2)(y)\dd x\dd y.\]
Using the definition of the cut metric,
we obtain that
\[\left|\;\int\limits_{[0,1]^2} h_i(x)V(x,y)g_j(y)\dd x\dd y\right|\le\cut{V}\]
for every $i,j\in\{1,2\}$.
The triangle inequality implies that
\[\|VUf\|_2^2
  \le\sum_{i,j\in [2]}\left|\int\limits_{[0,1]^2} h_i(x)V(x,y)g_j(y)\dd x\dd y\right|
  \le 4\cut{V},\]
which yields that $\|VUf\|_2\le 2\cut{V}^{1/2}$.
\end{proof}

We are now ready to show that
every sequence of digraphons convergent in the cut metric is also $\nu$-convergent;
we state the lemma more generally for kernels.

\begin{lemma}
\label{lm:converge}
If a sequence $(W_n)_{n\in\NN}$ of kernels satisfies $\|W_n\|_\infty\le 1$ for every $n\in \NN$ and converges to a kernel $W$ in the cut metric,
then the sequence $(W_n)_{n\in\NN}$ also $\nu$-converges to $W$.
\end{lemma}

\begin{proof}
Fix a sequence $(W_n)_{n\in\NN}$ of kernels and a digraphon $W$ as in the statement of the lemma and observe that $\|W_n\|_{2\to 2}\le 1$ for every $n\in \mathbb{N}$.
Since the sequence $(W_n)_{n\in\NN}$ converges in the cut metric,
for every $\varepsilon>0$,
there exists $n_0\in\NN$ such that $\cut{W-W_n}\le\varepsilon^2/4$ for every $n\ge n_0$.
In addition,
the Lebesgue density theorem yields that $\|W\|_\infty\le 2$,
which implies that $\|W-W_n\|_\infty\le 1$ for every $n\in\NN$.
Lemma~\ref{lm:cutnu} now yields that
$\|(W_n-W)W'\|_{2\to 2}\le\varepsilon$ for any kernel $W'$ with $\|W'\|_\infty\le 1$ and any $n\ge n_0$.
In particular, it holds that
\[\|(W_n-W)W\|_{2\to 2}\le\varepsilon\qquad\mbox{and}\qquad \|(W_n-W)W_n\|_{2\to 2}\le\varepsilon\]
for every $n\ge n_0$ and
so the sequence $(W_n)_{n\in\NN}$ $\nu$-converges to $W$.
\end{proof}

We are now ready to prove the next lemma,
which forms the core of the proofs of Theorems~\ref{thm:spectrum} and~\ref{thm:spectrumW}.

\begin{lemma}
\label{lm:spectrum}
Let $(W_n)_{n\in\NN}$ be a sequence of digraphons that converge in the cut metric to a digraphon $W$.
The spectrum $\Sp(W)$ is the unique closed set that is the limit of the spectra $\Sp(W_n)$ in the Hausdorff sense, and
for every non-zero $\lambda\in \Sp(W)$ we have that $m_W(\lambda)$ is eventually equal to the sum of the algebraic multiplicities of the eigenvalues in $\Sp(W_n)$ converging to $\lambda$,
i.e.~for every non-zero $\lambda\in \Sp(W)$, there exists $\varepsilon_0>0$ such that for every $0<\varepsilon<\varepsilon_0$ there is $n_0\in\NN$ such that
\[m_W(\lambda)=\sum_{\lambda'\in\Sp(W_n)\cap B_{\varepsilon}(\lambda)}m_{W_n}(\lambda')\]
for every $n\ge n_0$.
\end{lemma}

\begin{proof}
By Lemma~\ref{lm:converge}, the sequence $(W_n)_{n\in\NN}$ $\nu$-converges to $W$ and
so the sets $\Sp(W_n)$, $n\in\NN$, converge to $\Sp(W)$ in the Hausdorff sense by~\cite[Proposition~3.7]{SanD16}.
Since $W$ is a compact operator,
every non-zero element of $\Sp(W)$ is an eigenvalue of $W$ with finite algebraic multiplicity and is isolated in $\Sp(W)$.
Consequently, the set $\Sp(W)$ is closed as $0\in\Sp(W)$.
Lemma~\ref{lm:cor213} yields that
for every non-zero element $\lambda$ of $\Sp(W)$,
$m_{W}(\lambda)$
is eventually equal to the sum of the algebraic multiplicities of the eigenvalues in $\Sp(W_n)$ converging to $\lambda$,
i.e., if $\lambda$ is a non-zero element of $\Sp(W)$,
then there exists $\varepsilon_0>0$ such that for every $\varepsilon\in(0,\varepsilon_0)$
there exists $n_0\in\NN$ such that $\Sp(W_n)\cap B_{\varepsilon}(\lambda)$ is non-empty for every $n\ge n_0$.
\end{proof}

We are now ready to prove the main theorem of this section.

\begin{theorem}
\label{thm:spectrum}
Let $(G_n)_{n\in\NN}$ be a convergent sequence of digraphs and let $W$ be the limit digraphon.
The normalized spectra of $G_n$ converge to the spectrum $\Sp(W)$ in the Hausdorff sense.
Moreover, for every non-zero $\lambda\in \Sp(W)$ we have that $m_W(\lambda)$
is eventually equal to the sum of the algebraic multiplicities of the normalized eigenvalues in $\Sp(G_n)$ converging to $\lambda$.
\end{theorem}

\begin{proof}
Fix a convergent sequence $(G_n)_{n\in\NN}$ of digraphs.
As in the graph case~\cite[Lemma 5.3]{BorCLSV08},
it is possible to order vertices of $G_n$ in a way that
the sequence $(W_n)_{n\in\NN}$ of the associated step digraphons converges in the cut metric;
let $W_0$ be the digraphon that is the limit of $(W_n)_{n\in\NN}$ in the cut metric.
Since $W$ and $W_0$ are weakly isomorphic,
there exist measure preserving transformations $\varphi$ and $\psi$ such that
$W^\varphi$ and $W_0^\psi$ are equal almost everywhere.
It follows that $\Sp(W)=\Sp(W_0)$ and
$m_{W}(\lambda)=m_{W_0}(\lambda)$ for every $\lambda\in\Sp(W)\setminus\{0\}$.
Hence, it is enough to prove the statement of the theorem for the digraphon $W_0$.

Lemma~\ref{lm:spectrum} implies that
the spectrum of $W_0$ is the limit of the spectra of $W_n$ in the Hausdorff sense and
$m_{W_0}(\lambda)$
is eventually equal to the sum of the algebraic multiplicities of the eigenvalues in $\Sp(G_n)$ converging to $\lambda$.
Observe that the spectrum of $W_n$ is exactly the normalized spectrum of $G_n$ with $0$ possibly added, and
the algebraic multiplicity of every non-zero $\lambda\in\Sp(W_n)$
is equal to the algebraic multiplicity of $\lambda\cdot|G_n|$ in $\Sp(G_n)$.
Note that $\Sp(W_0)$ contains $0$ but $\Sp(G_n)$ need not contain $0$ unlike $\Sp(W_n)$.
Hence,
in order to show that the sets $\Sp(G_n)/|G_n|$, $n\in\NN$, converge to $\Sp(W_0)$ in the Hausdorff sense,
we need to show that for every $\varepsilon>0$,
there exists $n_0\in\NN$ such that $\Sp(G_n)/|G_n|\cap B_{\varepsilon}(0)$ is non-empty for every $n\ge n_0$

Consider $\varepsilon>0$ and let $X$ be the set of all $\lambda\in\Sp(W_0)$ with $|\lambda|\ge\varepsilon$.
Since every non-zero element of the closed set $\Sp(W_0)$ is isolated and $\|W_0\|_{2\to 2}\le 1$,
the set $X$ is finite and
there exists $\delta>0$ such that $\Sp(W_0)\cap B_{\delta}(\lambda)=\{\lambda\}$ for every $\lambda\in X$;
let $X'$ be the union of $B_{\delta/2}(\lambda)$, $\lambda\in X$ (note that the sets $B_{\delta/2}(\lambda)$
are disjoint by the choice of $\delta$).
Further, let $n_0$ be such that
\[m_{W_0}(\lambda)=\sum_{\lambda'\in\Sp(W_n)\cap B_{\delta/2}(\lambda)}m_{W_n}(\lambda')\]
for every $\lambda\in X$ and $n\ge n_0$.
If there are infinitely many $n\in\NN$ such that there exists $\lambda\in(\Sp(G_n)/|G_n|)\setminus X'$ with $|\lambda|\ge\varepsilon$,
then there exists $\lambda\not\in X$ with $|\lambda|\ge\varepsilon$ that
is a limit point of $\Sp(G_n)/|G_n|$ and so of $\Sp(W_n)$.
However, such $\lambda$ must be contained in $\Sp(W_0)$ (as it is the limit of $\Sp(W_n)$ in the Hausdorff sense),
which is impossible as $\lambda\not\in X$.
We conclude that there exists $n'_0\ge n_0$ such that
$(\Sp(G_n)/|G_n|)\setminus X'$ has no element $\lambda$ with $|\lambda|\ge\varepsilon$.

We now use that $X$ is finite and that
for every $\lambda\in X$,
$m_{W_0}(\lambda)$ is the sum of the algebraic multiplicities of eigenvalues converging to $\lambda$.
Since the sum of $m_{W_0}(\lambda)$ taken over $\lambda\in X$ is finite,
this implies that $\Sp(G_n)/|G_n|$ when $|G_n|$ is sufficiently large
must contain an eigenvalue $\lambda$ with $|\lambda|<\varepsilon$.
Formally, set $M$ to be the sum of $m_{W_0}(\lambda)$, $\lambda\in X$, and
let $n''_0\ge n'_0$ be such that every graph $G_n$, $n\ge n''_0$, has at least $M+1$ vertices.
Since the sum of algebraic multiplicities of $\lambda\in\Sp(G_n)\cap (|G_n|X')$ is exactly $M$ for every $n\ge n_0$,
the set $(\Sp(G_n)/|G_n|)\setminus X'$ has no element $\lambda$ with $|\lambda|\ge\varepsilon$ for every $n\ge n'_0$, and
$G_n$ has at least $M+1$ vertices for every $n\ge n''_0$,
it follows that $\Sp(G_n)/|G_n|$ contains $\lambda$ with $|\lambda|<\varepsilon$ for every $n\ge n''_0$.
In particular, $\Sp(G_n)/|G_n|\cap B_{\varepsilon}(0)$ is non-empty for every $n\ge n''_0$.
\end{proof}

As a digraphon analogy to Theorem~\ref{thm:spectrum},
it is possible to prove that the spectra of digraphons that converge in the cut distance
converge to the spectrum of a limit digraphon.

\begin{theorem}
\label{thm:spectrumW}
Let $(W_n)_{n\in\NN}$ be a sequence of digraphons that
converges to a digraphon $W$ in the cut distance.
Then, the spectra $\Sp(W_n)$ converge to the spectrum $\Sp(W)$ in the Hausdorff sense.
Moreover, for every non-zero $\lambda\in \Sp(W)$ we have that $m_W(\lambda)$ is eventually equal to the sum of the algebraic multiplicities of the eigenvalues in $\Sp(W_n)$ converging to $\lambda$.
\end{theorem}

\begin{proof}
Since the digraphons $(W_n)_{n\in\NN}$ converge to $W$ in the cut distance,
there exist measure preserving maps $\varphi_n:[0,1]\to [0,1]$, $n\in\NN$, such that
$\cut{W_n^{\varphi_n}-W}$ converges to zero.
Since the spectra of the digraphons $W_n$ and $W_n^{\varphi_n}$
are the same for every $n\in\NN$ (including the algebraic multiplicities of non-zero elements),
the statement of the theorem follows from Lemma~\ref{lm:spectrum}.
\end{proof}

\section{Cycle densities}
\label{sec:cycle}

In this section, we relate the spectra of digraphons and the density of cycles as given in the next theorem.

\begin{theorem}
\label{thm:cycle}
Let $W$ be a digraphon.
For every $\ell\ge 3$, it holds that
\[t(C_\ell,W)=\sum_{\lambda\in\Sp(W)\setminus\{0\}}m_W(\lambda)\cdot\lambda^\ell\]
where the sum is absolutely convergent.
\end{theorem}

\begin{proof}
Fix a convergent sequence $(G_n)_{n\in\NN}$ of digraphs such that
the digraphon $W$ is a limit of the sequence $(G_n)_{n\in\NN}$;
a sequence of $W$-random digraphs with increasing number of vertices has this property with probability one.
Observe that it holds that
\[t(C_{\ell},W)=\lim_{n\to\infty}t(C_\ell,G_n)=\lim_{n\to\infty}\frac{\sum\limits_{\lambda\in\Sp(G_n)}m_{G_n}(\lambda)\cdot\lambda^{\ell}}{|G_n|^{\ell}},\]
and we write $L$ for the value of this limit.

Consider $n\in\NN$, and let $A_n$ be the adjacency matrix of the digraph $G_n$.
Since all entries of $A_n^TA_n$ are non-negative integers between $0$ and $|G_n|$,
the trace of $A_n^TA_n$ is at most $|G_n|^2$.
In particular, the sum of squares of the singular values of $A_n$ is at most $|G_n|^2$.
Since the sum of the $p$-th powers of the singular values
is an upper bound on the sum of the absolute values of the $p$-th powers of the eigenvalues for any $p>0$,
we obtain that
\begin{equation}
\sum\limits_{\lambda\in\Sp(G)}m_{G_n}(\lambda)\cdot\left|\lambda^2\right|\le |G_n|^2
\label{eq:singular}
\end{equation}
for every $n\in\NN$.

We prove the following claim, which yields the statement of the theorem:

\medskip

\noindent\emph{For every $\ell\ge 3$ and $\varepsilon>0$,
there exists a finite set $S\subseteq\Sp(W)\setminus\{0\}$ such that it holds that
\[\left|\sum_{\lambda\in S'}m_W(\lambda)\cdot\lambda^\ell-L\right|\le\varepsilon\]
for any finite set $S'\subseteq\Sp(W)\setminus\{0\}$ that contains $S$, i.e., $S\subseteq S'$.}

\medskip

We now prove this claim.
As the claim is trivial when $W=0$, we may assume that $W\not= 0$.
Fix $\ell\ge 3$ and $\varepsilon>0$ (without loss of generality, we assume that $\varepsilon\in (0,1)$), and
choose the set $S$ to be the set containing those $\lambda\in\Sp(W)$ with $|\lambda|\ge\varepsilon/4$.
The set $S$ is finite since every non-zero element of the closed set $\Sp(W)$ is isolated in $\Sp(W)$ and $\|W\|_{2\to 2}\le 1$.
We may assume that the set $S$ is non-empty (otherwise, we consider smaller $\varepsilon$).
Next fix any finite set $S'\subseteq\Sp(W)\setminus\{0\}$ that contains $S$, 
set $M$ to be the sum of $m_W(\lambda)$ of all $\lambda\in S'$, and
choose $\delta\in (0,1)$ such that
$\delta\le\frac{\varepsilon}{4M\cdot2^\ell}$ and
the balls $B_{\delta}(\lambda)$ and $B_{\delta}(\lambda')$ are disjoint for any distinct $\lambda,\lambda'\in S'$.
Theorem~\ref{thm:spectrum} implies that there exists $n_0\in\NN$ such that
the following holds for all $n\ge n_0$:
\begin{itemize}
\item for every $\lambda\in\Sp(G_n)$ with $|\lambda|\ge\varepsilon|G_n|/2$,
      there exists $\lambda'\in S\subseteq S'$ such that $\lambda/|G_n|\in B_{\delta}(\lambda')$, and
\item for every $\lambda\in S'$,
      the sum of the algebraic multiplicities of all $\lambda'\in\Sp(G_n)\cap B_{\delta|G_n|}(\lambda|G_n|)$
      is equal to $m_W(\lambda)$.
\end{itemize}
Finally, fix any $n\ge n_0$ such that
\[\left|\frac{\sum\limits_{\lambda\in\Sp(G_n)}m_{G_n}(\lambda)\cdot\lambda^{\ell}}{|G_n|^\ell}-L\right|\;\le\;\frac{\varepsilon}{4},\]
and let $S''$ be the set of all $\lambda\in\Sp(G_n)$ such that $\lambda/|G_n|\in B_{\delta}(\lambda')$ for some $\lambda'\in S'$.
Note that every $\lambda\in\Sp(G_n)$ with $|\lambda|\ge\varepsilon|G_n|/2$ is contained in $S''$.

We obtain using $\eqref{eq:singular}$ that
\begin{align*}
\left|\sum_{\lambda\in\Sp(G_n)\setminus S''}m_{G_n}(\lambda)\cdot\lambda^{\ell}\right|
  & \le \sum_{\lambda\in\Sp(G_n)\setminus S''} m_{G_n}(\lambda)\cdot|\lambda^{\ell}|\\
  & \le \sum_{\lambda\in\Sp(G_n)\setminus S''} m_{G_n}(\lambda)\cdot\frac{\varepsilon^{\ell-2}}{2^{\ell-2}}|G_n|^{\ell-2}\cdot|\lambda^2|\\
  & \le \frac{\varepsilon}{2}|G_n|^{\ell-2}\cdot\sum_{\lambda\in\Sp(G_n)\setminus S''} m_{G_n}(\lambda)\cdot|\lambda^2|
    \le \frac{\varepsilon}{2}|G_n|^{\ell},
\end{align*}
which yields that
\begin{equation}
\left|\sum\limits_{\lambda\in S''}\frac{m_{G_n}(\lambda)\cdot\lambda^{\ell}}{|G_n|^\ell}-L\right|\;\le\;\frac{3\varepsilon}{4}.
\label{eq:limit}
\end{equation}
Observe that $|(x+y)^\ell-x^{\ell}|\le2^\ell\delta$ for any complex numbers $x$ and $y$ such that $|x|\le 1$ and $|y|\le\delta$, and
recall that $2^\ell\delta\le\frac{\varepsilon}{4M}$.
Hence, we obtain that
\begin{align*}
\left|\sum_{\lambda\in S''}\frac{m_{G_n}(\lambda)\cdot\lambda^{\ell}}{|G_n|^\ell}-\sum_{\lambda\in S'}m_W(\lambda)\cdot\lambda^{\ell}\right|
  & = \left|\sum_{\lambda\in S'}\sum_{\lambda'\in B_{\delta|G_n|}(\lambda|G_n|)}m_{G_n}(\lambda')\cdot\left(\frac{(\lambda')^{\ell}}{|G_n|^\ell}-\lambda^{\ell}\right)\right|\\
  & \le \left|\sum_{\lambda\in S'}\sum_{\lambda'\in B_{\delta|G_n|}(\lambda|G_n|)}m_{G_n}(\lambda')\cdot\frac{\varepsilon}{4M}\right|\\
  & = \left|\sum_{\lambda\in S'}m_W(\lambda)\cdot\frac{\varepsilon}{4M}\right| = \frac{\varepsilon}{4},
\end{align*}
which combines with \eqref{eq:limit} to
\[\left|\sum_{\lambda\in S'}m_W(\lambda)\cdot\lambda^{\ell}-L\right|\le\varepsilon.\]
This finishes the proof of the claim and so of the theorem, too.
\end{proof}

\section{Concluding remarks}
\label{sec:concl}

We briefly discuss the relation of the cycle density and spectral properties in the case of limits of digraphs that
may contain parallel edges oriented in the opposite way.
Following the standard line of arguments,
we obtain that the natural limit object representing convergent sequences of such digraphs
is a pair $(W_1,W_2)$ of measurable functions $W_1,W_2:[0,1]^2\to [0,1]$ that satisfy that
$W_1(x,y)=W_1(y,x)$ and $W_1(x,y)+W_2(x,y)+W_2(y,x)\le 1$ for all $(x,y)\in [0,1]^2$ and
the $n$-vertex $W$-random digraph is obtained as follows.
Choose $x_1,\ldots,x_n$ uniformly in $[0,1]$ and
include a pair of edges oriented in the opposite way between the $i$-th and $j$-th vertices with probability $W_1(x_i,x_j)=W_1(x_j,x_i)$,
an edge directed from the $i$-th vertex to the $j$-th vertex with probability $W_2(x_i,x_j)$,
an edge directed from the $j$-th vertex to the $i$-th vertex with probability $W_2(x_j,x_i)$, and
no edge between the $i$-th vertex and the $j$-th vertex with probability $1-W_1(x_i,x_j)-W_2(x_i,x_j)-W_2(x_j,x_i)$,
with the choices for different pairs $(i,j)$ mutually independent.
The homomorphism density of a digraph $H$ (without parallel edges) is then equal to
\begin{equation}
\int\limits_{[0,1]^{V(H)}}\prod_{uv\in E(H)}\left(W_1(x_u,x_v)+W_2(x_u,x_v)\right)\dd x_{V(H)}.
\label{eq:tHWg}
\end{equation}
Consider a function $W':[0,1]^2\to [0,1]$ defined as $W'(x,y)=W_1(x,y)+W_2(x,y)$;
note that $W'$ need not be a digraphon as $W'(x,y)+W'(y,x)$ may exceed one but $\|W'\|_{\infty}\le 1$.
As the integral in \eqref{eq:tHWg} matches the integral in \eqref{eq:tHW} for $W'$,
we can use the same line of reasoning as in Sections~\ref{sec:spectrum} and~\ref{sec:cycle}
to derive that the homomorphism density of $C_{\ell}$ in $(W_1,W_2)$,
is equal to
\[\sum_{\lambda\in\Sp(W')\setminus\{0\}}m_{W'}(\lambda)\cdot\lambda^\ell\]
for every $\ell\ge 3$.
However,
this derivation fails for $\ell=2$,
which represents the limit density of the two-vertex digraphs with two parallel edges oriented in the opposite way,
as \eqref{eq:tHW} and \eqref{eq:tHWg} match only for digraphs with no parallel edges;
in fact, the identity need not hold as evidenced by one of the examples given below.

To illustrate the above presented concepts, we work out in detail the following example.
Let $A_n$ be the adjacency matrix of a uniformly chosen $n$-regular undirected graph with $2n$ vertices.
Note that $A_n$ is a symmetric real matrix,
which implies that all eigenvalues of $A_n$ are real and
their algebraic and geometric multiplicities are the same;
let $\lambda_1\ge\cdots\ge\lambda_{2n}$ be the eigenvalues of $A_n$.
It holds that $\lambda_1=n$ and
there exists a constant $\alpha$ (see \cite{TikY19}) such that
$|\lambda_i|\le\alpha\sqrt{n}$ for all $i\in [2n]\setminus\{1\}$ with probability tending to one.

Let $H^1_n$ be the digraph with $4n$ vertices that
is created from two disjoint copies of the vertex set of $A_n$ and
adding a pair of parallel edges oriented in the opposite way
between a vertex in the first copy and a vertex in the second copy whenever the vertices are joined by an edge in $A_n$.
Similarly, $H^2_n$ is the digraph created from two disjoint copies of the vertex set of $A_n$
with a directed edge from a vertex in the first copy to a vertex in the second copy whenever such an edge is present in $A_n$, and
with a directed edge from a vertex in the second copy to a vertex in the first copy whenever such an edge is not present in $A_n$.
That is, the adjacency matrix of $H^1_n$ is $\begin{bmatrix} 0 & A_n \\ A_n & 0 \end{bmatrix}$, and
the adjacency matrix of $H^2_n$ is $\begin{bmatrix} 0 & A_n \\ J_n-A_n & 0 \end{bmatrix}$,
where $J_n$ is the $2n\times 2n$ all one matrix.
Suppose that $\lambda_1\ge\cdots\ge\lambda_{2n}$ are the eigenvalues of $A_n$, and
observe that
\begin{align*}
\Sp(H^1_n) & =\{n,-n,\lambda_2,\ldots,\lambda_n,-\lambda_2,\ldots,-\lambda_{2n}\}\\
\Sp(H^2_n) & =\{n,-n,\lambda_2\iota,\ldots,\lambda_n\iota,-\lambda_2\iota,\ldots,-\lambda_{2n}\iota\}
\end{align*}
where $\iota$ is the imaginary unit.
Note that the limit $W^1$ of the sequence $(H^1_n)_{n\in\NN}$ is given as
\[W^1_i(x,y)=\begin{cases}
             1/2 & \mbox{if $i=1$, $x\in [0,1/2]$ and $y\in [1/2,1]$,}\\
             1/2 & \mbox{if $i=1$, $x\in [1/2,1]$ and $y\in [0,1/2]$, and}\\
	     0 & \mbox{otherwise.}
             \end{cases}\]
Similarly, the limit $W^2$ of the sequence $(H^2_n)_{n\in\NN}$ is given as
\[W^2_i(x,y)=\begin{cases}
             1/2 & \mbox{if $i=2$, $x\in [0,1/2]$ and $y\in [1/2,1]$,}\\
             1/2 & \mbox{if $i=2$, $x\in [1/2,1]$ and $y\in [0,1/2]$, and}\\
	     0 & \mbox{otherwise.}
             \end{cases}\]
While the limits $W^1$ and $W^2$ of the sequences $(H^1_n)_{n\in\NN}$ and $(H^2_n)_{n\in\NN}$ are different,
they yield the same $W':[0,1]^2\to [0,1]$ (recall that $W'=W_1+W_2$).
Note that $\Sp(W')=\{1/4,-1/4,0\}$, $m_{W'}(1/4)=1$ and $m_{W'}(-1/4)=1$, and
$\Sp(W')$ is also the limit of the normalized spectra of either $H^1_n$ or $H^2_n$.
Also note that the homomorphism density of $C_\ell$, $\ell\ge 3$, in either of $W^1$ and $W^2$
is equal $4^{-\ell}+(-4)^{-\ell}=2^{-2\ell+1}$ in line with the discussion above,
which also matches the limit homomorphism density of $C_\ell$ in either $H^1_n$ or $H^2_n$.
On the other hand, the limit of $t(C_2,H^1_n)$ is $1/4$ but $t(C_2,H^2_n)$ is $0$,
where $C_2$ is the two-vertex digraph with two parallel edges oriented in the opposite way.
%It follows that
%the homomorphism density of $C_2$ cannot be determined using spectral properties of the limit representation.

\section*{Acknowledgement}

The authors would like to thank Igor Balla for discussions concerning the spectra of random graphs and
Svante Janson for discussions concerning the convergence of digraphons in cut metric and in $\|{-}\|_{2\to 2}$.
The authors would also like to thank the two anonymous reviewers for their insightful comments.

\bibliographystyle{bibstyle}
\bibliography{cyclelimit}

\end{document}